\documentclass[11pt,a4paper]{amsart}
\usepackage{pdfsync}
\usepackage{mathptmx} 
\usepackage[scaled=0.90]{helvet} 
\usepackage{courier} 
\normalfont
\usepackage[T1]{fontenc}

\usepackage{amssymb}
\usepackage{marvosym}
\usepackage{graphicx}
\usepackage{epstopdf}
\usepackage{lscape}
\usepackage{pifont}
\usepackage{tikz-cd}

\usepackage[curve,tips]{xypic}
\SelectTips{eu}{11}
\UseTips

\usepackage{hyperref}

\renewcommand{\epsilon}{\varepsilon}

\newtheorem{theorem}{Theorem}[section]
\newtheorem{proposition}[theorem]{Proposition}
\newtheorem{corollary}[theorem]{Corollary}
\newtheorem{lemma}[theorem]{Lemma}

\theoremstyle{definition}

\theoremstyle{remark}

\newcommand{\CC}{\mathfrak C}

\newcommand{\Z}{\mathbb Z}

\newcommand{\C}{\mathbb C}

\newcommand{\cohom}[3]{H^{{\raise1pt\hbox{$\scriptstyle#1$}}}(#2\>\!,#3)}
\newcommand{\tatecohom}[3]%
  {\widehat H^{{\raise1pt\hbox{$\scriptstyle#1$}}}(#2\>\!,#3)}

\newcommand{\Cohom}[3]%
  {H^{{\raise1pt\hbox{$\scriptstyle#1$}}}\big(#2\>\!,#3\big)}
\newcommand{\Tatecohom}[3]%
  {\widehat H^{{\raise1pt\hbox{$\scriptstyle#1$}}}\big(#2\>\!,#3\big)}

\newcommand{\homol}[3]{H_{{\lower1pt\hbox{$\scriptstyle#1$}}}(#2\>\!,#3)}
\newcommand{\homolog}[2]{H_{{\lower1pt\hbox{$\scriptstyle#1$}}}(#2)}

\title{The first $\ell^2$-Betti number and groups acting on trees}
\author{Indira Chatterji}
\author{Sam Hughes}
\author{Peter Kropholler}
\address{Universit\'e de Nice, Parc Valrose, 06000 Nice, FRANCE}
\address{School of Mathematical Sciences, University of Southampton, Southampton, SO17 1BJ, United Kingdom}

\email{indira.chatterji@math.cnrs.fr}
\email{sam.hughes@soton.ac.uk}
\email{p.h.kropholler@soton.ac.uk}
\date{\today} 

\begin{document}
\maketitle

\begin{abstract}
We generalise results of Thomas, Allcock, Thom--Petersen, and Kar--Niblo to the first $\ell^2$-Betti number of quotients of certain groups acting on trees by subgroups with free actions on the edge sets of the graphs.
\end{abstract}
\section{Introduction}
The $\ell^2$-Betti numbers $b^{(2)}_i(G)$ of a group $G$ are defined in \cite{Lueck2002}.
The $\ell^2$-Euler characteristic $\chi^{(2)}$ of $G$ is the alternating sum of these Betti numbers and is denoted $\chi^{(2)}{(G)}$.
Let $\CC$ denote the class of groups $F$ such that 
\begin{itemize}
\item $b^{(2)}_1(F)=b^{(2)}_2(F)=0$, and 
\item either $\chi^{(2)}(F)=0$ or $F$ is finite.
\end{itemize}
Note that that $\CC$ contains all $\ell^2$-acyclic groups (i.e. the groups for which $b_i^{(2)}=0$ for all $i>0$) and in particular it contains all amenable groups.   Relevant background on $\ell^2$-cohomology is included in Section~\ref{sec.background}. In this note we prove the following theorem.

\begin{theorem}\label{thm.main} Let $F$ be a group acting cocompactly on a tree, with vertex and edge stabilisers in $\CC$, let $N$ be a subgroup normally generated by $m$ elements, intersecting the vertex stabilisers trivially. Let $G$ denote $F/N$ and set $k:=\chi^{(2)}(F)+m$.
Then the following conclusions hold:
\begin{enumerate}
    \item If $k \leq 0$, then $G$ is infinite. \label{thm.main.1}
    \item If $k < 0$, then $ b^{(2)}_1(G)\geq -k >0$.   \label{thm.main.2}
    \item If $G$ is finite, then $k>0$ 
and $|G|\geq \frac{1}{k}$. \label{thm.main.4}
\end{enumerate}
\end{theorem}
Note that the hypotheses of this theorem guarantee that $N$ acts freely on the specified tree and in particular $N$ is necessarily a free group.  Note also that, according to \cite[Corollary~1.4]{Bader_et_al}, if $ b^{(2)}_1(G)>0$ then $G$ has no commensurated infinite amenable subgroup and according to \cite[Corollary 6]{BekkaValette} does not have property (T).  If we also have $b_2^{(2)}(G)=0$, then $G$ is in the class $\mathcal D_{\rm reg}$ by \cite[Lemma~2.8]{Thom}. We refer the reader to \cite{BHV} for background on property (T) and to \cite[Definition 2.6]{Thom} for the definition of the class $\mathcal D_{\mathrm{reg}}$. 
By the main result of Osin's paper \cite{osin} we have the following corollary. 
\begin{corollary}
Let $G$, $F$ and $N$ be as in Theorem \ref{thm.main}. Assume that $G$ is finitely presented, (virtually) indicable and that $\chi^{(2)}(F)+m<0$.  Then $G$ is (virtually) acylindrically hyperbolic.
\end{corollary}
%

The simplest way in which the indicability hypothesis may arise is through \emph{stable letters}:
Let $T$ denote the $F$-tree of Theorem \ref{thm.main}. Let $K$ denote the (necessarily normal) subgroup generated by the vertex stabilisers. Then there is a subgroup $E\le F$ that complements $K$ and all such subgroups are free of uniquely determined rank. Such a subgroup may be referred to as \emph{a subgroup of stable letters of the action}. The group $G$ has an infinite cyclic quotient when $N\cap E$ has infinite index in $E$, in other words when there is a stable letter that is faithfully represented in $G$, and in this case $G$ is indicable.

Recall that a group $G$ is \emph{$C^\ast$-simple} if the reduced group $C^\ast$-algebra, denoted $C_r^\ast(G)$, has exactly two norm closed $2$-sided ideals, $0$, and the algebra $C_r^\ast(G)$ itself.  By \cite[Corollary~6.7]{BreuillardKalantarKennedyOzawa2017} we obtain the following.
\begin{corollary}
With $G$, $F$ and $N$ as before, $G$ is $C^\ast$-simple if and only if it has trivial amenable radical.
\end{corollary}
Theorem \ref{thm.main} has some historical pedigree.  It originally began life as a result about quotients of free groups due to Thomas (see Theorem~\ref{thm.classic}\eqref{thomas-allcock}) and was proved using combinatorial methods \cite{thomas1988}.  The result was generalised by Allcock to incorporate a bound on the rank of the abelianisation of the quotient group \cite{Allcock1999}.  The introduction of $\ell^2$-cohomology came when Peterson--Thom \cite[Theorem~3.6]{PetersonThom} and Kar--Niblo \cite{karniblo} independently linked the inequality of Thomas to the first $\ell^2$-Betti number. These discoveries are summarized in the following result.
\begin{theorem}
[Thomas \cite{thomas1988}, Allcock \cite{Allcock1999}, Peterson--Thom \cite{PetersonThom}, Kar--Niblo \cite{karniblo}] 
\label{thm.classic}
Let $G$ be a group with a presentation
\[\langle x_{1},\dots,x_{n};\ r_{1}^{k_{1}},\dots,r_{m}^{k_{m}}\rangle\]
in which the elements $r_{i}$ have order $k_{i}$ when interpreted in $G$.
\begin{enumerate}
\item \label{thomas-allcock}
If $n-\sum_{i=1}^{m}\frac1{k_{i}}\ge1$ then $G$ is infinite.
\item
If $G$ is finite then $|G|\ge\frac1{1-n+\sum_{i=1}^{m}k_{i}}$.
\item
If $n-\sum_{i=1}^{m}\frac1{k_{i}} > 1$ then $G$ is non-amenable.
\end{enumerate}
\end{theorem}
\begin{proof}[Deduction of Theorem~\ref{thm.classic} from Theorem~\ref{thm.main}]
Let $G$ be a group with a presentation 
\[G=\langle x_{1},\dots,x_{n}|\ r_{1}^{k_{1}},\dots,r_{m}^{k_{m}}\rangle.\]
Adding $m$ fresh generators $y_{1},\dots,y_{m}$, we can give the following alternative presentation of the same group:
\[G=\langle x_{1},\dots,x_{n},y_{1},\dots,y_{m}|\ y_{1}^{k_{1}},\dots,y_{m}^{k_{m}}, r_{1}y_{1}^{-1},\dots,r_{m}y_{m}^{-1}\rangle.\]
Let $F$ be the group with presentation
\[F=\langle x_{1},\dots,x_{n},y_{1},\dots,y_{m}|\ y_{1}^{k_{1}},\dots,y_{m}^{k_{m}}\rangle\] 
and let $N$ be the subgroup of $F$ normally generated by $r_{1}y_{1}^{-1},\dots,r_{m}y_{m}^{-1}$. Then $F$ is a free product of cyclic groups: in particular it is virtually free and has Euler characteristic $\chi(F)=\sum_{i=1}^{m}\frac1{k_{i}}-n-m+1$. The condition that the $r_{i}$ have order $k_{i}$ in the original presentation ensures that $N$ does not meet any of the finite subgroups of $F$ and so is torsion-free. Applying Theorem \ref{thm.main} with these choices of $F$, $N$, $G$ yields Theorem \ref{thm.classic}. 
\end{proof}
Finally, we also provide a computation of the first $\ell^2$-Betti number for certain groups acting on trees. This generalises a result of L\"uck \cite{l08appendix}, which covers the case of an amalgamated free product, and a result of Tsouvalas \cite[Corollary~3.7]{t18}. Tsouvalas assumes the vertex stabilisers are either residually finite or virtually torsion-free and the edge stabilisers are finite. Here we replace both of these assumptions with L\"uck's less restrictive assumption that the first $\ell^{2}$-Betti numbers of the edge stabilisers vanish.  So, for example, the theorem applies to fundamental groups of graphs of $\CC$-groups.
\begin{theorem}\label{thm.firstL2} 
Let $F$ be a group acting on a tree and let $V$ and $E$ denote sets of representatives of $F$-orbits of vertices and edges.  Assume for each $e\in E$ that $ b^{(2)}_1(F_e)=0$, then
\[ b^{(2)}_1(F)=\sum_{v\in V}\left( b^{(2)}_1(F_v)- \frac{1}{|F_v|}\right)+\sum_{e\in E} \frac{1}{|F_e|}, \]
where for a group $G$, $\frac{1}{|G|}$ is interpreted as $0$ if $G$ is infinite.
\end{theorem}
\subsection*{Acknowledgements}
The second author wishes to thank Ian Leary for his tireless and good humoured role as thesis adviser.  The second author was supported by the Engineering and Physical Sciences Research Council grant number 2127970.
\section{Background on \texorpdfstring{${\ell^2}$}{l^2}-homology} \label{sec.background}
Let $G$ be a group. Then both $G$ and the complex group algebra $\C G$ act by left multiplication on the Hilbert space $\ell^2G$ of square-summable sequences. The group von Neumann algebra $\mathcal NG$ is the ring of $G$-equivariant bounded operators on $\ell^2G$. The regular elements of $\mathcal NG$ form an Ore set and the Ore localization of $\mathcal NG$ can be identified with the \emph{ring of affiliated operators}, and is denoted by $\mathcal UG$. One has the inclusions $\C G\subseteq \mathcal N G\subseteq\ell^2G\subseteq\mathcal UG$ and it is also known that $\mathcal UG$ is a self-injective ring which is flat over $\mathcal NG$. For more details concerning these constructions we refer the reader to \cite{Lueck2002} and especially to Theorem 8.22 of \S8.2.3 therein. 
The \emph{von Neumann dimension} and the basic properties we need can be found in \cite[\S8.3]{Lueck2002}.
Now let $Y$ be a $G$-CW complex as defined in \cite[Definition 1.25 of \S1.2]{Lueck2002}. The $\ell^2$-homology groups of $Y$ are then defined to be the equivariant homology groups $H^G_i(Y;\mathcal UG)$, and we have
\[
 b^{(2)}_i(Y)=\dim_{\mathcal UG} H^G_i(Y;\mathcal UG).
\]
The $\ell^2$-Betti numbers of a group $G$ are then defined to be the $\ell^2$-Betti numbers of $EG$.  By \cite[Theorem 6.54(8)]{Lueck2002}, the zeroth $\ell^2$-Betti number of $G$ is equal to $1/|G|$ where $1/|G|$ is defined to be zero if $G$ is infinite.  Moreover, if $G$ is finite then $b_n^{(2)}(G)=0$ for $n\geq 1$.

Let $C_*(Y;\mathcal UG)$ denote the standard cellular chain complex of $Y$ with coefficients in $\mathcal UG$.  We have the formula 
\[
\dim_{\mathcal UG}C_i(Y;\mathcal UG)=\sum_\sigma\frac1{|G_\sigma|}
\]
where $\sigma$ runs through a set of orbit representatives of $i$-dimensional cells in $Y$ and if $G_\sigma$ is infinite then $1/|G_\sigma|$ is taken to be equal to $0$. Standard arguments of homological algebra give the connection between two Euler characteristic computations:
\begin{equation}\label{eqn.eulerchar}
\sum_i(-1)^i b^{(2)}_i(Y)=\sum_i(-1)^i\dim_{\mathcal UG}C_i(Y;\mathcal UG).
\end{equation}
We record two other consequences, both of which can be found in \cite[Section~6.6]{Lueck2002}.
\begin{lemma}\label{lem.eulerchar}
Let $G$ be a group and let $Y$ be a cocompact $G$-complex with finite stabilizers.  Then we have 
\begin{enumerate}
\item
$\displaystyle\chi^{(2)}(G)=\sum_{i=0}^{\infty}\left((-1)^{i}\sum_{v\in S_{i}}\frac1{|G_{v_{i}}|}\right)=\sum_{i=0}^\infty(-1)^i b^{(2)}_i(G);$
\item
if $G$ has a finite index subgroup $H$ with a finite $K(H,1)$ then 
\[\chi^{(2)}(G)=\chi(G).\]
\end{enumerate}
\end{lemma}
We will need the following lemma for the proofs in the next section.  One should think of it as a mild generalisation of Theorem~6.54(2) in \cite{Lueck2002}
\begin{lemma}[Comparison with the Borel construction up to rank]\label{lem.Borel}
Let $X$ be a $G$-CW complex.  Suppose for all $x\in X$ the isotropy group $G_x$ is finite or $ b^{(2)}_p(G_x)=0$ for all $0\leq p \leq n$, then
\[ b^{(2)}_p(X)= b^{(2)}_p(EG\times X)\quad \text{for } 0\leq p \leq n. \]
\end{lemma}
\begin{proof}
It suffices to prove that the dimension of the kernel and cokernel of the map
\[\textrm{pr}_p:H^G_p(EG\times X;\mathcal{N}G)\rightarrow H^G_p(X;\mathcal{N}G) \]
induced by the projection $EG\times X\rightarrow X$ are trivial for $0\leq p \leq n$.  By an identical argument to \cite[Theorem~6.54(2)]{Lueck2002} it suffices to prove for each isotropy subgroup $H\leq G$ and $0\leq p \leq n$ the kernel and cokernel of the map $\textrm{pr}_p:H^H_p(EH;\mathcal{N}H)\rightarrow H^H_p(\ast;\mathcal{N}H)$ have trivial dimension.  If $H$ is finite this follows from \cite[Theorem~6.54(8a)]{Lueck2002}, and is immediate if $ b^{(2)}_p(H)=0$ for all $0\leq p \leq n$.
\end{proof}
\section{The Main Theorem}
To prove Theorem \ref{thm.main}, one needs the following method of computing the $\ell^2$-Euler characteristic of a group acting on a tree analogous to Chiswell's result \cite{Chiswell} for rational Euler characteristic.
\begin{proposition}[Chatterji--Mislin \cite{ChatterjiMislin}] \label{prop.EulerChar}
Let $F$ be a group acting on a tree and let $V$ and $E$ denote sets of representatives of $F$-orbits of vertices and edges.  If the $\ell^2$-Euler characteristic of each vertex and edge group is finite, then
\[\chi^{(2)}(F)=\sum_{v\in V}\chi^{(2)}(F_v)-\sum_{e\in E}\chi^{(2)}(F_e). \]
\end{proposition}
\begin{proof}[Proof of Theorem \ref{thm.main}]
There is a cocompact action of $F$ on a tree $T$ with vertex and edge stabilizers in the class $\CC$. Let $V$ and $E$ denote the vertex and edge sets.
Let $\overline T$ denote the quotient graph $T/N$ and write $\overline V$ and $\overline E$ for its vertex and edge sets. Now $G=F/N$ acts cocompactly on $\overline T$ with vertex and edge stabilizers in $\CC$. The augmented chain complex of $T$ is the short exact sequence
\[0\to\Z E\to \Z V\to\Z\to0.\]
Upon factoring out the action of $N$ we have the exact sequence
\begin{equation}\label{eqn.homology.les}H_{1}(N,\Z)\to\Z\overline E\to\Z\overline V\to\Z\to0,\end{equation}
this being the tail end of the long exact sequence of homology of $N$.
Let $\{r_i\colon i=1,..m\}$ denote a normal generating set for $N$.  Choose a vertex $v_{0}$ in $T$ to be a fixed basepoint. For $1\le i\le m$ consider the geodesic from $v_{0}$ to $v_{0}r_{i}$. In the quotient graph $\overline T$ this geodesic descends to a loop because $v_{0}$ and $v_{0}r_{i}$ become identified in $\overline T$. Now $2$-discs can be glued to each loop. By adjoining free $G$-orbits of $2$-discs equivariantly we can build a $2$-complex $Y$ with an action of $G$, whose $1$-skeleton is $\overline T$, and which has augmented cellular chain complex
\begin{equation}\label{eqn.augchaincomplex}\Z G^{m}\to\Z \overline E\to\Z\overline V\to\Z\to0.\end{equation}
The exactness of \eqref{eqn.homology.les} ensures the exactness of \ref{eqn.augchaincomplex}.
 
Let $V_{0}$ and $E_{0}$ be sets of orbit representatives of vertices and edges in $Y$.
Now, applying \cite[Theorem 6.80(1)]{Lueck2002},  then\eqref{eqn.eulerchar} and finally Proposition \ref{prop.EulerChar}, we have that
\begin{align*}
\chi^{(2)}(Y)&= b^{(2)}_{0}(Y)- b^{(2)}_{1}(Y)+ b^{(2)}_{2}(Y)\\
\sum_{v\in V_{0}}\frac1{|G_{v}|}-\sum_{e\in E_{0}}\frac1{|G_{e}|}+m 
&= b^{(2)}_{0}(Y)- b^{(2)}_{1}(Y)+ b^{(2)}_{2}(Y)\\
\chi^{(2)}(F)+m&= b^{(2)}_{0}(Y)- b^{(2)}_{1}(Y)+ b^{(2)}_{2}(Y).\\
\intertext{Applying Lemma~\ref{lem.Borel}, we have}
\chi^{(2)}(F)+m&= b^{(2)}_{0}(EG\times Y)- b^{(2)}_{1}(EG\times Y)+ b^{(2)}_{2}(EG\times Y).\\
&\geq b^{(2)}_{0}(EG\times Y)- b^{(2)}_{1}(EG\times Y).\\
\intertext{Applying \cite[Theorem~6.54(1a)]{Lueck2002} to the diagonal map $EG\rightarrow EG\times Y$ we see that}
\chi^{(2)}(F)+m&\geq b^{(2)}_{0}(G)- b^{(2)}_{1}(G).
\end{align*}
Let $k=\chi^{(2)}(F)+m$. If $k\leq 0$, then $b^{(2)}_{0}(G)- b^{(2)}_{1}(G)\leq 0$ and so $G$ is infinite, this proves \eqref{thm.main.1}.  Now, assume $k<0$. In this case $G$ is infinite and therefore $b^{(2)}_{0}(G)=0$. It follows that $b^{(2)}_{1}(G)\geq -k>0$, this proves \eqref{thm.main.2}.

If $G$ is finite, then $b^{(2)}_0(G)=\frac{1}{|G|}$, $b^{(2)}_1(G)=0$, and $k>0$.  In particular, $k\geq \frac{1}{|G|}>0$ and \eqref{thm.main.4} follows.
\end{proof}
\section{On the \texorpdfstring{$\ell^2$}{l^2}-invariants for certain groups acting on trees}
\begin{proof}[Proof of Theorem~\ref{thm.firstL2}]
Let $V$ and $E$ denote sets of representatives of $F$-orbits of vertices and edges for the action of $F$ on the tree.  We consider the relevant part of the $E^1$-page for the $F$-equivariant spectral-sequence (see Chapter~VII.9 of \cite{brownbook}) applied to the tree:
\[\begin{tikzcd}
{}  &   &   &   \\
1  & \bigoplus_{v\in V}H_1^F(F\times_{F_v}EF_v;\mathcal{U}F) & 0  &  \\
0 &   \bigoplus_{v\in V}H_0^F(F\times_{F_v}EF_v;\mathcal{U}F))  & \arrow[l, "d^1"]\bigoplus_{e\in E}H_0^F(F\times_{F_e}F_e;\mathcal{U}F)) & \\
{} \arrow[uuu, shift right=8] \arrow[rrr, shift left=8] & 0 & 1 & {}
\end{tikzcd}\]
If $F$ is finite then $b_1^{(2)}(F)=0$, so $d^1$ is injective and $E^2_{1,0}=0$.  The result follows from the fact $E^1_{0,1}=0$.  

Now, assume $F$ is infinite, then $d^1$ is surjective since $ b^{(2)}_0(F)=0$.  Thus, \[\dim_{\mathcal{U}F}(\mathrm{Ker}(d^1))=\sum_{e\in E} b^{(2)}_0(F_e)-\sum_{v\in V} b^{(2)}_0(F_v).\]
Now, the spectral sequence obviously collapses on the $E^2$-page and $E^1_{0,1}=E^2_{0,1}$.  Since von Neumann dimension is additive over short exact sequences, we have 
\begin{align*}
b^{(2)}_1(F)&=\dim_{\mathcal{U}F}(\mathrm{Ker}(d^1))+\dim_{\mathcal{U}F}(E_{0,1}^2)\\
&= \left(\sum_{e\in E} b^{(2)}_0(F_e) -  \sum_{v\in V} b^{(2)}_0(F_v) \right) + \sum_{v\in V} b^{(2)}_1(F_v),
\end{align*}
and the result follows.
\end{proof}
\bibliography{bib}

\begin{thebibliography}{10}

\bibitem{Allcock1999}
D.~Allcock.
\newblock Spotting infinite groups.
\newblock {\em Math. Proc. Cambridge Philos. Soc.}, 125(1):39--42, 1999.

\bibitem{Bader_et_al}
U.~Bader, A.~Furman, and R.~Sauer.
\newblock Weak notions of normality and vanishing up to rank in
  {$L^2$}-cohomology.
\newblock {\em Int. Math. Res. Not. IMRN}, (12):3177--3189, 2014.

\bibitem{BHV}
B.~Bekka, P.~de~la Harpe, and A.~Valette.
\newblock {\em Kazhdan's property ({T})}, volume~11 of {\em New Mathematical
  Monographs}.
\newblock Cambridge University Press, Cambridge, 2008.

\bibitem{BekkaValette}
M.~E.~B. Bekka and A.~Valette.
\newblock Group cohomology, harmonic functions and the first {$L^2$}-{B}etti
  number.
\newblock {\em Potential Anal.}, 6(4):313--326, 1997.

\bibitem{BreuillardKalantarKennedyOzawa2017}
E.~Breuillard, M.~Kalantar, M.~Kennedy, and N.~Ozawa.
\newblock {$C^*$}-simplicity and the unique trace property for discrete groups.
\newblock {\em Publ. Math. Inst. Hautes \'{E}tudes Sci.}, 126:35--71, 2017.

\bibitem{brownbook}
K.~S. Brown.
\newblock {\em Cohomology of groups}, volume~87 of {\em Graduate Texts in
  Mathematics}.
\newblock Springer-Verlag, New York, 1982.

\bibitem{l08appendix}
N.~P. Brown, K.~J. Dykema, and K.~Jung.
\newblock Free entropy dimension in amalgamated free products.
\newblock {\em Proc. Lond. Math. Soc. (3)}, 97(2):339--367, 2008.
\newblock With an appendix by Wolfgang L\"{u}ck.

\bibitem{ChatterjiMislin}
I.~Chatterji and G.~Mislin.
\newblock Hattori-{S}tallings trace and {E}uler characteristics for groups.
\newblock In {\em Geometric and cohomological methods in group theory}, volume
  358 of {\em London Math. Soc. Lecture Note Ser.}, pages 256--271. Cambridge
  Univ. Press, Cambridge, 2009.

\bibitem{Chiswell}
I.~M. Chiswell.
\newblock Exact sequences associated with a graph of groups.
\newblock {\em J. Pure Appl. Algebra}, 8(1):63--74, 1976.

\bibitem{karniblo}
A.~Kar and G.~Niblo.
\newblock Some non-amenable groups.
\newblock {\em Pub. Mat.}, 56(1):255--259, 2012.

\bibitem{Lueck2002}
W.~L{\"u}ck.
\newblock {\em {$L^2$}-invariants: theory and applications to geometry and
  {$K$}-theory}, volume~44 of {\em Ergebnisse der Mathematik und ihrer
  Grenzgebiete. 3. Folge. A Series of Modern Surveys in Mathematics [Results in
  Mathematics and Related Areas. 3rd Series. A Series of Modern Surveys in
  Mathematics]}.
\newblock Springer-Verlag, Berlin, 2002.

\bibitem{osin}
D.~Osin.
\newblock On acylindrical hyperbolicity of groups with positive first
  {$\ell^2$}-{B}etti number.
\newblock {\em Bull. Lond. Math. Soc.}, 47(5):725--730, 2015.

\bibitem{PetersonThom}
J.~Peterson and A.~Thom.
\newblock Group cocycles and the ring of affiliated operators.
\newblock {\em Invent. Math.}, 185:561--592, 2011.

\bibitem{Thom}
A.~Thom.
\newblock Low degree bounded cohomology and {$L^2$}-invariants for negatively
  curved groups.
\newblock {\em Groups Geom. Dyn.}, 3(2):343--358, 2009.

\bibitem{thomas1988}
R.~M. Thomas.
\newblock Cayley graphs and group presentations.
\newblock {\em Math. Proc. Cambridge Philos. Soc.}, 103(3):385--387, 1988.

\bibitem{t18}
K.~Tsouvalas.
\newblock Euler characteristics on virtually free products.
\newblock {\em Comm. Algebra}, 46(8):3397--3412, 2018.

\end{thebibliography}
\bibliographystyle{abbrv}

\end{document}